\documentclass[12pt,oneside]{amsart}

\usepackage{setspace,geometry}
\newtheorem{lem}{Lemma}
\newtheorem{thm}{Theorem}
\newtheorem{cor}{Corollary}

\usepackage{filecontents}
\usepackage{paralist}

\usepackage[backend=biber, datamodel=mrnumber, giveninits=true, url=false, maxbibnames=9, sorting=nyt]{biblatex}
\usepackage{hyperref}
\usepackage{tikz-cd}

\DeclareSourcemap{
  \maps[datatype=bibtex]{
    \map{
      \step[fieldsource=doi, final]
      \step[fieldset=isbn, null]
      \step[fieldset=issn, null]
    }
  }
}

\DeclareFieldFormat*{title}{\mkbibemph{#1}}
\DeclareFieldFormat*{journaltitle}{#1}
\DeclareFieldFormat*{booktitle}{#1}

\renewbibmacro{in:}{%
  \ifentrytype{article}
    {}
    {\printtext{\bibstring{in}\intitlepunct}}}

\DeclareFieldFormat[article,periodical]{volume}{\mkbibbold{#1}}
\DeclareFieldFormat[article,periodical]{number}{\bibstring{number}~#1}
\renewbibmacro*{journal+issuetitle}{%
  \usebibmacro{journal}%
  \setunit*{\addspace}%
  \iffieldundef{series}
    {}
    {\newunit
     \printfield{series}%
     \setunit{\addspace}}%
  \printfield{volume}%
  \setunit{\addspace}%
  \usebibmacro{issue+date}%
  \setunit{\addcolon\space}%
  \usebibmacro{issue}%
  \setunit{\addcomma\space}%
  \printfield{number}%
  \setunit{\addcomma\space}%
  \printfield{eid}%
  \newunit}
  
\DeclareFieldFormat[article,periodical]{pages}{#1}

\DeclareFieldFormat{mrnumber}{%
  MR\addcolon\space
  \ifhyperref
    {\href{https://mathscinet.ams.org/mathscinet/article?mr=#1}{\nolinkurl{#1}}}
    {\nolinkurl{#1}}}
    
\renewbibmacro*{doi+eprint+url}{%
  \iftoggle{bbx:doi}
    {\printfield{doi}}
    {}%
  \newunit\newblock
  \printfield{mrnumber}%
  \newunit\newblock
  \iftoggle{bbx:eprint}
    {\usebibmacro{eprint}}
    {}%
  \newunit\newblock
  \iftoggle{bbx:url}
    {\usebibmacro{url+urldate}}
    {}}

\begin{document}
\title[The trace simplex of a noncommutative Villadsen algebra]{The trace simplex of a\\ noncommutative Villadsen algebra}
\author{George~A. Elliott}
\address{Department of Mathematics\\ University of Toronto\\ Toronto, ON, Canada\ \ M5S 2E4}
\email{elliott@math.toronto.edu}
\author{Vincent~M. Ruzicka}
\address{Department of Mathematics and Statistics\\
University of Wyoming\\
Laramie, WY 82071\\
USA}
\email{vruzicka@uwyo.edu}
\date{\today}

\subjclass[2020]{}

\begin{abstract}
	We construct a ``noncommutative'' Villadsen algebra $B$ and show that, given an extreme tracial state $\nu$ on its canonical AF subalgebra, the subset of $T(B)$ consisting of those tracial states that equal $\nu$ when restricted to the canonical AF subalgebra is the Poulsen simplex. In particular, if the canonical AF subalgebra has a unique trace, then $T(B)$ is the Poulsen simplex. We go on to show that in certain instances, the tracial cone of a ``classical'' AF-Villadsen algebra $D$ is isomorphic to the tracial cone of the algebra obtained from $D$ by deleting all point evaluations. 
\end{abstract}
\maketitle

\section{Introduction}\label{S:intro}

	Villadsen algebras (of the first type) were introduced in \cite{jV98} (they are not to be confused with those of the second type, introduced in \cite{jV99}). Progress on the classification of these algebras was made in \cite{gaEcgLzN24} and \cite{gaEzN25-Vill}. Moreover, in \cite[Theorem 4.5]{gaEcgLzN24} it was shown that the simplex of tracial states of a Villadsen algebra is the Poulsen simplex when the seed space is not a single point. In the present paper, we construct ``noncommutative'' Villadsen algebras and deduce from our main result (Theorem \ref{T:main}) that, under certain conditions (see the final paragraph of this section), the simplex of tracial states of such an algebra is also the Poulsen simplex. 
	
	An example of this ``noncommutative'' construction is as follows. Let $C_0$ be a nuclear unital C*-algebra, and consider the inductive sequence 
\begin{align}\label{E:introSequenceUHF}
	M_2(C_0) \xrightarrow{\phi_1} M_4(C_0\otimes C_0) \xrightarrow{\phi_2} M_8\big (C_0^{\otimes 4} \big ) \xrightarrow{\phi_3} \cdots,
\end{align}
where the seed for the $i$-th stage map $\phi_i$ is
\begin{align*}
	C_0^{\otimes 2^{i-1}} \ni c \mapsto \begin{pmatrix}
		c\otimes 1&0\\
		0&1\otimes c
	\end{pmatrix} \in M_2\big ( C_0^{\otimes 2^{i}}\big ),
\end{align*}
and $1$ denotes the unit of $C_0^{\otimes 2^{i-1}}$. In the case that $C_0$ is commutative, i.e., $C_0 = C(X)$, using the usual identification of $C(X) \otimes C(X)$ with $C(X^2)$ one sees that the limit $B$ of \eqref{E:introSequenceUHF} is a Villadsen algebra---nonsimple unless $X$ is a point, since we have not introduced point evaluations. On the other hand, when $C_0$ is noncommutative, we call $B$ a \textit{noncommutative} Villadsen algebra. In analogy with the commutative case, we call $C_0$ the ``seed algebra'' of $B$. 

	More generally, in this paper we consider noncommutative AF-Villadsen algebras, i.e., limits of finite direct sums of matrix algebras over tensor powers of a seed algebra (examples of traditional ``commutative'' AF-Villadsen algebras were given in \cite{iHncP26}, and a classification result for such algebras with a fixed well-behaved seed space was obtained in \cite{gaEzN25-Vill}). 
	
	For example, consider the limit $B$ of the inductive sequence
\begin{align*}
	M_2(C_0) \oplus M_2(C_0) \xrightarrow{\phi_1} M_4(C_0\otimes C_0) \oplus M_4(C_0\otimes C_0) \xrightarrow{\phi_2} M_8\big (C_0^{\otimes 4} \big ) \oplus M_8\big (C_0^{\otimes 4} \big ) \xrightarrow{\phi_3} \cdots,
\end{align*}
where $\phi_i$ is defined by
\begin{multline*}
	C_0^{\otimes 2^{i-1}}\oplus C_0^{\otimes 2^{i-1}} \ni (c_1,c_2) \mapsto \\ (\begin{pmatrix}
		c_1\otimes 1&0\\
		0&1\otimes c_2
	\end{pmatrix}, \begin{pmatrix}
		c_1\otimes 1&0\\
		0&1\otimes c_2
	\end{pmatrix}) \in M_2\big ( C_0^{\otimes 2^{i}}\big ) \oplus M_2\big ( C_0^{\otimes 2^{i}}\big ). 
\end{multline*}
Then $B$ is a noncommutative AF-Villadsen algebra, with seed algebra $C_0$. As in the commutative case, we use the term ``noncommutative Villadsen algebra'' (without the ``AF-'' prefix) to describe this more general construction as well. Notice that $B$ contains as a subalgebra the limit $A$ of the inductive sequence
\begin{align*}
	M_2(\mathbb C) \oplus M_2(\mathbb C) \xrightarrow{\phi_1'} M_4(\mathbb C) \oplus M_4(\mathbb C) \xrightarrow{\phi_2'} M_8 (\mathbb C) \oplus M_8 (\mathbb C ) \xrightarrow{\phi_3'} \cdots,
\end{align*}
where $\phi_i'$ is the restriction of $\phi_i$ to $M_{2^i}(\mathbb C) \oplus M_{2^i}(\mathbb C)$. We call $A$ the canonical AF subalgebra of $B$. 

	It follows from our main result that if the seed algebra of a given noncommutative Villadsen algebra $B$ has more than one trace and if the canonical AF subalgebra of $B$ is simple and has a unique trace, then the simplex of tracial states on $B$ is the Poulsen simplex, i.e., the unique simplex for which the extreme points are dense (Corollary \ref{C:main}). 
	
	This paper proceeds as follows. In the next section, we define noncommutative Villadsen algebras and show that one may assume all connecting maps are canonical in the sense of Equation \eqref{E:composed}. Then, letting $B$ denote a noncommutative Villadsen algebra and $A$ denote its canonical AF subalgebra, in Section \ref{S:trace} we show that each trace on $A$ extends to a trace on $B$ (Lemma \ref{L:traceExtends}), and furthermore, that if $\nu$ is an extreme trace on $A$, then the fiber over $\nu$ (i.e., the set of traces on $B$ which restrict to $\nu$ on $A$) is a face of $T(B)$ (Lemma \ref{L:fiberIsCompact}). In the final section of this paper, we prove the main result, and we extend the main result to the AF-Villadsen algebras of \cite{gaEzN25-Vill}; we accomplish this by showing that if $D$ is an AF-Villadsen algebra and $B$ is the AF-Villadsen algebra obtained from $D$ by deleting all point evaluations, then under the assumption of uniform convergence of a certain sequence in the affine function space over the state space of the $K_0$-group of the canonical AF subalgebra of $B$, the tracial cones of $D$ and $B$ are isomorphic (and in some cases, the tracial simplexes themselves are isomorphic; see Theorem \ref{T:traceIntertwining}). 

\section{A noncommutative Villadsen algebra construction}\label{S:AFVill}

	Consider the following inductive sequence of ordered abelian groups with distinguished order units: 
\begin{align}\label{E:dimensionGroupSequence}
	(G_i,\theta_i)_{i\in \mathbb N},\quad G_i = \big (\mathbb Z^{j_i},\mathbb Z^{j_i}_{+},(n_{i,1},\dotsc,n_{i,j_i}) \big ),
\end{align}
where $\theta_i$ is determined by the multiplicity matrix $[\theta_{i;k,l}]$, $1\leq k \leq j_{i}$, $1 \leq l \leq j_{i+1}$. Assume that an AF algebra whose K-theory is given by the limit of the above sequence is infinite-dimensional. 

	For each $i \in \mathbb N$ and $1\leq l \leq j_{i+1}$, choose disjoint sets $P_{i;1,l},P_{i;2,l},\dotsc,P_{i;j_i,l}$ that partition the set of integers $\{1,2,\dotsc,n_{i+1,l}\}$ and are such that $|P_{i;k,l}|=\theta_{i;k,l}n_{i,k}$, i.e.,
\begin{align*}
	\bigsqcup_{k=1}^{j_i} P_{i;k,l} = \{1,2,\dotsc,n_{i+1,l}\}, \quad 1 \leq l \leq j_{i+1};
\end{align*}
then partition each $P_{i;k,l}$ into disjoint sets $P_{i;k,l}^{(1)},P_{i;k,l}^{(2)},\dotsc,P_{i;k,l}^{(\theta_{i;k,l})}$ each of cardinality $n_{i,k}$, and fix an enumeration $p_{i;k,l}^{(m,1)},p_{i;k,l}^{(m,2)},\dotsc,p_{i;k,l}^{(m,n_{i,k})}$ of each $P_{i;k,l}^{(m)}$. Denote the set of these sets with the fixed enumerations by $\mathcal P$, i.e.,
\begin{align}\label{E:partition}
	\mathcal P = \Big \{P_{i;k,l}^{(m)} = \{p_{i;k,l}^{(m,1)},\dotsc,p_{i;k,l}^{(m,n_{i,k})}\} \mid i \in \mathbb N, 1 \leq k \leq j_i, 1 \leq l \leq j_{i+1}, 1 \leq m \leq \theta_{i;k,l}\Big \};
\end{align}
let us call $\mathcal P$ a partition for $(G_i,\theta_i)$. 

	Now let $C_0$ be a nuclear unital C*-algebra, and construct a unital C*-algebra $B( (G_i,\theta_i),C_0,\mathcal P)$ as follows. For each $i \in \mathbb N$, let 
\begin{align*}
	C_{i,k} = C_0^{\otimes n_{i,k}},\ B_{i,k} = M_{n_{i,k}}(\mathbb C)\otimes C_{i,k} \cong M_{n_{i,k}}(C_{i,k}),\quad \ 1 \leq k \leq j_i,
\end{align*}
and let
\begin{align}\label{E:buildingBlock}
	B_i = \bigoplus_{k=1}^{j_i} B_{i,k} = \bigoplus_{k=1}^{j_i} M_{n_{i,k}} (C_{i,k}  ) = \bigoplus_{k=1}^{j_i} M_{n_{i,k}} \big (C_0^{\otimes n_{i,k}}  \big ).
\end{align}
Define the seed of an injective unital *-homomorphism $\phi_{i} \colon B_i \to B_{i+1}$ (up to unitary equivalence) by
\begin{multline}\label{E:connectingMap}
	\bigoplus_{k=1}^{j_i} C_0^{\otimes n_{i,k}} \ni \bigoplus_{k=1}^{j_i} \bigotimes_{t=1}^{n_{i,k}} c_t^{(k)} \mapsto \\ \bigoplus_{l=1}^{j_{i+1}} \text{diag} \big ( \bigotimes_{s=1}^{n_{i+1,l}} d_s^{(1,l,1)},\dotsc,\bigotimes_{s=1}^{n_{i+1,l}} d_s^{(1,l,\theta_{i;1,l})},\dotsc, \bigotimes_{s=1}^{n_{i+1,l}} d_s^{(j_i,l,1)},\dotsc,\bigotimes_{s=1}^{n_{i+1,l}} d_s^{(j_i,l,\theta_{i;j_i,l})}\big )   \\ \in \bigoplus_{l=1}^{j_{i+1}} M_{\theta_{i;1,l}+\theta_{i;2,l}+\cdots+\theta_{i;j_i,l}}\big (C_0^{\otimes n_{i+1,l}}\big ),
\end{multline}
where
\begin{align*}
	d_{s}^{(k,l,m)} = \begin{cases}
		c_t^{(k)},& s=p_{i;k,l}^{(m,t)}\\
		1,& s \not \in P_{i;k,l}^{(m)}
	\end{cases},\quad 1 \leq s \leq n_{i+1,l},\ 1 \leq m \leq \theta_{i;k,l},\ 1 \leq l \leq j_{i+1},\ 1 \leq k \leq j_i.
\end{align*}
Then define $B((G_i,\theta_i),C_0,\mathcal P)$ to be the limit of the inductive sequence $(B_i,\phi_i)_{i\in \mathbb N}$. 

	As alluded to in Section \ref{S:intro}, in the case that the seed algebra $C_0$ is commutative, this construction yields a traditional ``commutative'' Villadsen algebra. 
On the other hand, if the seed algebra is noncommutative, then $B((G_i,\theta_i),C_0,\mathcal P)$ will be called a \textit{noncommutative} Villadsen algebra. It turns out that $B((G_i,\theta_i),C_0,\mathcal P)$ is independent of the partition $\mathcal P$ for $(G_i,\theta_i)$ as the following lemma shows. Hence, we may write $B((G_i,\theta_i),C_0,\mathcal P) = B((G_i,\theta_i),C_0)$. 

\begin{lem}\label{L:composed}
	Let $(G_i,\theta_i)_{i\in \mathbb N}$ be as in Equation \eqref{E:dimensionGroupSequence}, $\mathcal P$ be as in Equation \eqref{E:partition}, 
\begin{align*}
	\mathcal Q = \Big \{Q_{i;k,l}^{(m)} = \{q_{i;k,l}^{(m,1)},\dotsc,q_{i;k,l}^{(m,n_{i,k})}\} \mid i \in \mathbb N, 1 \leq k \leq j_i, 1 \leq l \leq j_{i+1}, 1 \leq m \leq \theta_{i;k,l}\Big \}
\end{align*}
be another partition for $(G_i,\theta_i)$, and $C_0$ be a nuclear unital C*-algebra. Then 
\begin{align*}
	B((G_i,\theta_i),C_0,\mathcal P) \cong B((G_i,\theta_i),C_0,\mathcal Q).
\end{align*}  
\end{lem}

\begin{proof}
	Let $B((G_i,\theta_i),C_0,\mathcal P)$ and $B((G_i,\theta_i),C_0,\mathcal Q)$ denote the limits of the sequences $(B_i,\phi_i)_{i\in \mathbb N}$ and $(B_i,\psi_i)_{i\in \mathbb N}$, with $B_i$ defined as in Equation \eqref{E:buildingBlock} and $\phi_i$ and $\psi_i$ defined according to Equation \eqref{E:connectingMap}. Fix $i \in \mathbb N$, and let $\sigma_k$ be a permutation of $\{1,\dotsc,n_{i,k}\}$ for each $1\leq k \leq j_i$; then $\sigma_1,\dotsc,\sigma_{j_i}$ together induce a *-isomorphism $\sigma\colon B_i\to B_{i}$ with seed
\begin{align*}
	\bigoplus_{k=1}^{j_i} C_0^{\otimes n_{i,k}} \ni \bigoplus_{k=1}^{j_i} \bigotimes_{t=1}^{n_{i,k}} c_t^{(k)} \mapsto \bigoplus_{k=1}^{j_i} \bigotimes_{t=1}^{n_{i,k}} c_{\sigma_k(t)}^{(k)} \in \bigoplus_{k=1}^{j_i} C_0^{\otimes n_{i,k}}.
\end{align*}

	Keeping $i$ fixed, let $\gamma_l$ be a permutation of $\{1,2,\dotsc,n_{i+1,l}\}$ for each $1 \leq l \leq j_{i+1}$ such that $\gamma_l(q_{i;k,l}^{(m,t)}) = p_{i;k,l}^{(m,\sigma_k(t))}$ for each $1 \leq t \leq n_{i,k}$, $1 \leq m \leq \theta_{i;k,l}$, and $1 \leq k \leq j_i$ (and such that when $s \not \in Q_{i;k,l}$, $\gamma_l(s) \not \in P_{i;k,l}$). Then a straightforward calculation shows that the isomorphism $\gamma\colon B_{i+1} \to B_{i+1}$ induced by $\gamma_1,\dotsc,\gamma_{j_{i+1}}$ with seed
\begin{align*}
	\bigoplus_{l=1}^{j_{i+1}} C_0^{\otimes n_{i+1,l}} \ni \bigoplus_{l=1}^{j_{i+1}} \bigotimes_{t=1}^{n_{i+1,l}} d_t^{(l)} \mapsto \bigoplus_{l=1}^{j_{i+1}} \bigotimes_{t=1}^{n_{i+1,l}} d_{\gamma_l(t)}^{(l)} \in \bigoplus_{l=1}^{j_{i+1}} C_0^{\otimes n_{i+1,l}}
\end{align*}
makes the diagram
\[\begin{tikzcd}[cramped]
	{B_i} & {B_{i+1}}\\
	{B_i} & {B_{i+1}}
	\arrow["{\phi_i}", from=1-1, to=1-2]
	\arrow["{\sigma}"', from=1-1, to=2-1]
	\arrow["{\gamma}", from=1-2, to=2-2]
	\arrow["{\psi_i}"', from=2-1, to=2-2]
\end{tikzcd}\]
commute. 

	It follows that one may choose a sequence of isomorphisms $\beta_i\colon B_i \to B_i$, each induced by permutations of $\{1,2,\dotsc,n_{i,k}\}$ respectively, such that the diagram
\[\begin{tikzcd}
	{B_1} & {B_2} & {B_3} & \cdots \\
	{B_1} & {B_2} & {B_3} & \cdots
	\arrow["{\phi_1}", from=1-1, to=1-2]
	\arrow["{\beta_1}"', from=1-1, to=2-1]
	\arrow["{\phi_2}", from=1-2, to=1-3]
	\arrow["{\beta_2}"', from=1-2, to=2-2]
	\arrow["{\phi_3}", from=1-3, to=1-4]
	\arrow["{\beta_3}"', from=1-3, to=2-3]
	\arrow["{\psi_1}"', from=2-1, to=2-2]
	\arrow["{\psi_2}"', from=2-2, to=2-3]
	\arrow["{\psi_3}"', from=2-3, to=2-4]
\end{tikzcd}\]
commutes. This proves that $B((G_i,\theta_i),C_0,\mathcal P) \cong B((G_i,\theta_i),C_0,\mathcal Q)$ as asserted.	
\end{proof}

	Note that in the case of a noncommutative UHF-Villadsen algebra, such as the one given by the limit of Equation \eqref{E:introSequenceUHF}, this lemma is almost obvious since one need only permute diagonal elements to go from one partition to another. 

	Fix $i,t \in \mathbb N$, and denote the multiplicity matrix for the composed map 
\begin{align*}
	\theta_{i,i+t-1} = \theta_{i+t-1}\circ \cdots \circ \theta_{i}\colon G_i \to G_{i+t}
\end{align*} 
by $[\theta_{i,i+t-1;k,l}]$, $1\leq k \leq j_i$, $1 \leq l \leq j_{i+t}$. Notice that $\theta_{i,i;k,l} = \theta_{i;k,l}$ and
\begin{align*}
	\theta_{i,i+t-1;k,l} = \sum_{m=1}^{j_{i+t-1}} \theta_{i+t-1;m,l} \theta_{i,i+t-2;k,m}.
\end{align*} 
The takeaway from Lemma \ref{L:composed} is that we may assume the composed map $\phi_{i,i+t-1}\colon B_{i} \to B_{i+t}$ is canonical in the sense that the seed is of the form
\begin{align}\label{E:composed}
	\bigoplus_{k=1}^{j_i} C_{i,k} \ni (c_1,&\dotsc,c_{j_i}) \mapsto \\ \notag \bigoplus_{l=1}^{j_{i+t}} \text{diag} \Big ( &\underbrace{c_1 \otimes 1_{i,1}\otimes \cdots \otimes 1_{i,1}}_{\theta_{i,i+t-1;1,l}}\otimes \cdots \otimes \underbrace{1_{i,j_i}\otimes \cdots \otimes 1_{i,j_i}}_{\theta_{i,i+t-1;j_i,l}},\dotsc,\\  \notag &\quad \underbrace{1_{i,1}\otimes \cdots \otimes 1_{i,1}\otimes c_1}_{\theta_{i,i+t-1;1,l}} \otimes \cdots \otimes \underbrace{1_{i,j_i}\otimes \cdots \otimes 1_{i,j_i}}_{\theta_{i,i+t-1;j_i,l}},\dotsc,\\ \notag
	&\qquad \qquad\underbrace{1_{i,1}\otimes \cdots \otimes 1_{i,1}}_{\theta_{i,i+t-1;1,l}}\otimes \cdots \otimes \underbrace{c_{j_i}\otimes 1_{i,j_i} \otimes \cdots \otimes 1_{i,j_i}}_{\theta_{i,i+t-1;j_i,l}},\dotsc,\\ \notag &\qquad \qquad\quad \underbrace{1_{i,1}\otimes \cdots \otimes 1_{i,1}}_{\theta_{i,i+t-1;1,l}}\otimes \cdots \otimes \underbrace{1_{i,j_i} \otimes \cdots \otimes 1_{i,j_i} \otimes c_{j_i}}_{\theta_{i,i+t-1;j_i,l}} \Big ) \\ &\qquad \qquad \qquad \qquad \qquad \qquad \qquad  \in \bigoplus_{l=1}^{j_{i+t}} M_{\theta_{i,i+t-1;1,l}+\cdots+\theta_{i,i+t-1;j_i,l}}\big (C_{i+t,l}\big ),\notag
\end{align}
where $1_{i,k}$ is the identity of $C_{i,k}$. 

\section{The trace simplex}\label{S:trace}

	 Given a convex subset $K$ of a topological vector space, denote the set of its extreme points by $\partial K$ and its closure by $\overline{K}$. In this section, fix a C*-algebra $B := B((G_i,\theta_i),C_0)$ obtained from the inductive sequence $(B_i,\phi_i)_{i\in \mathbb N}$, where $(G_i,\theta_i)_{i\in \mathbb N}$ is as in Equation \eqref{E:dimensionGroupSequence}, $B_i$ is as in Equation \eqref{E:buildingBlock}, $\phi_i$ is as in Equation \eqref{E:connectingMap}, and $C_0$ is a nuclear unital C*-algebra. In this paper, when discussing traces on a unital C*-algebra, we mean tracial \textit{states}. 

	The (Choquet) simplex of tracial states, or trace simplex, $T(B)$ of $B$ is (affinely homeomorphic to) the limit of the affine projective system
\begin{align*}
	T(B_1)\xleftarrow{\phi_1^*} T(B_2) \xleftarrow{\phi_2^*}T(B_{3}) \xleftarrow{\phi_3^*} \cdots,\quad \phi_{i}^*(\tau) = \tau \circ \phi_i.
\end{align*}
Hence, a trace $\tau \in T(B)$ is uniquely represented by a sequence $(\tau_i)_{i\in \mathbb N}$ with $\tau_i \in T(B_i)$ and $\phi_i^*(\tau_{i+1}) = \tau_i$; moreover, for each $i \in \mathbb N$, there exist scalars $0 \leq \lambda_1^{(i)},\dotsc,\lambda_{j_i}^{(i)} \leq 1$ summing to one such that 
\begin{align*}
	\tau_{i} = \lambda_1^{(i)} \tau_1^{(i)} + \cdots + \lambda_{j_i}^{(i)} \tau_{j_i}^{(i)}, \quad \tau_{k}^{(i)} \in T(C_{i,k})
\end{align*}
(note that we are making the canonical identifications of $T(C_{i,k})$ with $T(B_{i,k})$ and of $T(B_{i,k})$ with $T(B_{i,k})\circ \pi_k \subseteq T(B_i)$). In this way, we associate to $\tau$ a sequence of tuples of scalars $((\lambda_{1}^{(i)},\dotsc,\lambda_{j_i}^{(i)}))_{i\in \mathbb N}$ and a sequence of tuples of traces $((\tau_{1}^{(i)},\dotsc,\tau_{j_i}^{(i)}))_{i\in \mathbb N}$. If we wish to specify this information when discussing $\tau$, we shall write $\tau = (\tau_i;\lambda_k^{(i)},\tau_k^{(i)})$. 

	Using Lemma \ref{L:composed} (more specifically Equation \eqref{E:composed}), a calculation reveals that, for any $t \in\mathbb N$, the composed map $\phi_{i,i+t-1}^*\colon T(B_{i+t}) \to T(B_i)$ has a seed of the form 
\begin{align}\label{E:traceImage}
	\lambda_{1}^{(i+t)}\tau_{1}^{(i+t)}+\cdots+\lambda_{j_{i+t}}^{(i+t)}\tau_{j_{i+t}}^{(i+t)} \mapsto \sum_{l=1}^{j_{i+t}} \sum_{k=1}^{j_i}  \sum_{m=1}^{\theta_{i,i+t-1;k,l}} \frac{\lambda_l^{(i+t)} n_{i,k}}{n_{i+t,l}} \mu_{k,l,m}  = \lambda_{1}^{(i)}\tau_{1}^{(i)}+\cdots+\lambda_{j_{i}}^{(i)}\tau_{j_{i}}^{(i)},
\end{align}
where $\mu_{k,l,m} \in T(C_{i,k})$ is defined by
\begin{align}\label{E:traceImageParts}
	\mu_{k,l,m} (c ) = \tau_l^{(i+t)}\Big (\bigotimes_{s=1}^{k-1} \big ( 1_{i,s}^{\otimes \theta_{i,i+t-1;s,l}} \big ) \otimes 1_{i,k}^{\otimes (m-1)}  \otimes  c \otimes  1_{i,k}^{\otimes (\theta_{i,i+t-1;k,l}-m)}  \otimes \bigotimes_{s=k+1}^{j_i} \big (1_{i,s}^{\otimes \theta_{i,i+t-1;s,l}} \big ) \Big )
\end{align}
(notice that $\phi_i^* = \phi_{i,i}^*$). It follows that (assuming $\lambda_k^{(i)} \not = 0$)
\begin{align}\label{E:traceImageTraces}
	\tau_k^{(i)} = \frac{1}{\lambda_k^{(i)}}\sum_{l=1}^{j_{i+t}}\sum_{m=1}^{\theta_{i,i+t-1;k,l}} \frac{\lambda_l^{(i+t)} n_{i,k}}{n_{i+t,l}} \mu_{k,l,m} ,
\end{align}
and hence
\begin{align}\label{E:coefficientImage}
	\lambda_{k}^{(i)} = \sum_{l=1}^{j_{i+t}} \frac{\lambda_l^{(i+t)}n_{i,k}\theta_{i,i+t-1;k,l}}{n_{i+t,l}}. 
\end{align}

	It is clear that $B$ contains as a subalgebra the limit $A$ of the inductive sequence
\begin{align*}
	\bigg ( \bigoplus_{k=1}^{j_i}M_{n_{i,k}}(\mathbb C), \phi_i' \bigg )_{i\in \mathbb N},
\end{align*}
where $\phi_i'$ is the restriction of $\phi_i$ to $\bigoplus_{1\leq k \leq j_i}M_{n_{i,k}}(\mathbb C)$. We call $A$ the \textit{canonical AF subalgebra} of $B$. Of course, a trace $\nu \in T(A)$ is specified by a triple $(\nu_i;\alpha_{k}^{(i)},\text{Tr}_k^{(i)})$, where $\text{Tr}_k^{(i)}$ denotes the (normalized) trace on $M_{n_{i,k}}(\mathbb C)$.

\begin{lem}\label{L:traceExtends}
	Let $\nu = (\nu_i;\alpha^{(i)}_k,\text{Tr}^{(i)}_k) \in T(A)$, and let $T(C_0)$ be nonempty. Then there exists a trace in $T(B)$ whose associated sequence of tuples of scalars is $((\alpha^{(i)}_1,\dotsc,\alpha^{(i)}_{j_i}))_{i\in \mathbb N}$. 
	
	In particular, each trace on $A$ extends to a trace on $B$. 
\end{lem}

\begin{proof}
	By assumption, there exist traces $\tau^{(1)}_1,\dotsc,\tau^{(1)}_{j_{1}}$ in $T(C_{1,1}),\dotsc,T(C_{1,j_{1}})$, respectively. Define the trace
\begin{align*}
	\tau_1 := \alpha_1^{(1)} \tau^{(1)}_1 +\cdots + \alpha_{j_{1}}^{(1)}\tau^{(1)}_{j_{1}} \in T(B_1).
\end{align*}
Now, for each $i \in \mathbb N$, recursively define the trace
\begin{align*}
	\tau_{i+1} := \alpha_1^{(i+1)} \tau^{(i+1)}_{1} +\cdots + \alpha_{j_{i+1}}^{(i+1)}\tau^{(i+1)}_{j_{i+1}}\in T(B_{i+1}),
\end{align*}
where
\begin{align*}
	\tau^{(i+1)}_{l} = (\tau_1^{(i)})^{\otimes \theta_{i,i;1,l}}  \otimes \cdots \otimes (\tau_{j_{i}}^{(i)})^{\otimes \theta_{i,i;j_{i},l}} \in T(C_{i+1,l}). 
\end{align*}
It then follows from Equations \eqref{E:traceImageParts}, \eqref{E:traceImageTraces}, and \eqref{E:coefficientImage} that $\phi_{i,i}^{*}(\tau_{i+1}) = \tau_i$.

	We have thus constructed a trace $\tau = (\tau_i;\alpha_k^{(i)}, \tau_k^{(i)}) \in T(B)$ whose associated sequence of scalars is $((\alpha^{(i)}_1,\dotsc,\alpha^{(i)}_{j_i}))_{i\in \mathbb N}$. This proves the first statement. For the second statement, observe that $\tau|_{A} = \nu$. 
\end{proof}

	Recall that $\tau \in T(B)$ has a base of neighborhoods consisting of sets of the form $\{\tau': |\tau(b)-\tau'(b)|<\epsilon,\ b \in \mathcal F\}$, where $\epsilon >0$ and $ \mathcal F \subset B$ is finite; we shall denote such a neighborhood of $\tau$ by $\mathcal N(\epsilon,\mathcal F)$. We will use the following simple corollary of the Krein-Milman theorem, which we state without proof, in our main result. 

\begin{lem}\label{L:kM}
	Let $K$ be a compact convex subset of $T(B)$, let $\tau \in K$, and let $\mathcal N\subseteq K$ be a basic neighborhood of $\tau$. Then there is a number $N \in \mathbb N$ such that for all $n > N$, there exist points $\tau_1,\dotsc,\tau_n \in \partial K$ such that $n^{-1}(\tau_1+\cdots+\tau_n) \in \mathcal N$. 
\end{lem}

	Let $\nu = (\nu_i;\alpha^{(i)}_k,\text{Tr}^{(i)}_k) \in T(A)$. Denote by $F_{\nu}$ the \textit{fiber over $\nu$}; that is,
\begin{align*}
	F_{\nu} = \{\tau \in T(B) : \tau|_A = \nu\}.
\end{align*}
Notice that $\tau = (\tau_i)_{i\in \mathbb N} \in F_{\nu}$ if and only if its associated sequence of tuples of scalars is $((\alpha^{(i)}_1,\dotsc,\alpha^{(i)}_{j_i}))_{i\in \mathbb N}$ if and only if there exists an $i \in \mathbb N$ such that $\tau_i = \alpha^{(i)}_1\tau^{(i)}_1+\cdots + \alpha^{(i)}_{j_i}\tau^{(i)}_{j_i}$ for some $\tau^{(i)}_k \in T(C_{i,k})$, $1 \leq k \leq j_i$. Also, by Lemma \ref{L:traceExtends}, $F_{\nu}$ is nonempty when $T(C_0)$ is nonempty. 

\begin{lem}\label{L:fiberIsCompact}
	Let $\nu = (\nu_i;\alpha^{(i)}_k,\text{Tr}^{(i)}_k) \in T(A)$. Then $F_{\nu}$ is a compact convex subset of $T(B)$. 
	
	Moreover, if $\nu \in  \partial T(A)$, then $F_{\nu}$ is a face of $T(B)$. 
\end{lem}

\begin{proof}
	If $F_\nu$ is empty, then so is $T(C_0)$ by Lemma \ref{L:traceExtends}, hence so is $T(B)$, and the results follow trivially. If $F_\nu$ is a singleton, the results are still trivial. So suppose $F_\nu$ contains at least two traces. 
	
	A short calculation shows that 
\begin{align*}
	(1-\lambda)\tau + \lambda \mu =\Big ( (1-\lambda)\tau_i + \lambda \mu_i; \alpha_k^{(i)}, (1-\lambda)\tau_k^{(i)} + \lambda \mu_k^{(i)} \Big )
\end{align*}
for any $\tau = (\tau_i;\alpha^{(i)}_k,\tau^{(i)}_k), \mu = (\mu_i;\alpha^{(i)}_k,\mu^{(i)}_k) \in F_\nu$ and $\lambda \in [0,1]$. That is, $F_\nu$ is convex. 

	Furthermore, $F_{\nu}$ is closed, hence compact. For this, let $(\tau_{\beta})_{\beta \in \Lambda}$ be a net in $F_\nu$ converging to $\tau = (\tau_i;\lambda^{(i)}_k,\tau^{(i)}_k) \in T(B)$, where $\tau_\beta = (\tau_{\beta,i};\alpha^{(i)}_k,\tau^{(i)}_{\beta,k})$. Then for any finite subset $\mathcal F \subset B$ and any $\epsilon > 0$, there exists $\gamma \in \Lambda$ such that 
\begin{align}\label{E:neighborhood}
	|\tau(f)-\tau_\beta(f)| < \epsilon,\quad \forall f \in \mathcal F,\ \beta \geq \lambda.
\end{align}
Taking $f = (0,\dotsc,0,1,0,\dotsc,0) \in B_i$, we see that 
\begin{align}\label{E:evaluationAtIdentity}
	|\tau_i(f) - \tau_{\beta,i}(f)| = |\lambda_k^{(i)} - \alpha_k^{(i)}|.
\end{align}
Since $\epsilon$ can be chosen to be arbitrarily small, it follows from Equations \eqref{E:neighborhood} and \eqref{E:evaluationAtIdentity} that $\alpha^{(i)}_k = \lambda^{(i)}_k$ for each $i \in \mathbb N$, $1 \leq k \leq j_i$; hence $\tau \in F_\nu$. 

	For the second statement, let $\tau,\mu \in T(B)$ and suppose that $(1-\lambda)\tau + \lambda \mu \in F_\nu$ for some $0<\lambda <1$; then 
\begin{align*}
	\nu = \big((1-\lambda)\tau + \lambda \mu \big )|_A = (1-\lambda)\tau|_A + \lambda \mu |_A
\end{align*}
so that, by the hypothesis of the statement, $\tau|_A = \mu|_A = \nu$. That is, $\tau,\mu \in F_{\nu}$. 
\end{proof}

Letting $F$ be a face of $T(B)$, recall that $F$ is in particular a simplex. Moreover, notice that $F$ is obtained from the limit of the affine projective system
\begin{align*}
	F_1\xleftarrow{\phi_1^*|_{F_2}} F_2 \xleftarrow{\phi_2^*|_{F_3}}F_3 \xleftarrow{\phi_3^*|_{F_4}} \cdots,
\end{align*}
where $F_i$ is a face of $T(B_i)$. Hence, when $\nu = (\nu_i;\alpha^{(i)}_k,\text{Tr}^{(i)}_k) \in \partial T(A)$, by Lemma \ref{L:fiberIsCompact}, we have that $F_\nu$ is obtained from the limit of the projective sequence $(F_i, \phi_{i}^*|_{F_{i+1}})_{i\in \mathbb N}$, where 
\begin{align*}
	F_i = \{\alpha^{(i)}_1\tau^{(i)}_1+\cdots + \alpha^{(i)}_{j_i}\tau^{(i)}_{j_i} \mid \tau^{(i)}_k \in T(C_{i,k}),\, 1\leq k \leq j_i\}.
\end{align*}

	The last lemma of this section characterizes some of the extreme points of $F_{\nu}$. 

\begin{lem}\label{L:limitOfExtremeTraces}
	Let $\nu = (\nu_i;\alpha^{(i)}_k,\text{Tr}^{(i)}_k) \in T(A)$ and $\tau = (\tau_i;\alpha_k^{(i)},\tau_k^{(i)}) \in F_\nu$. If for $i$ sufficiently large, $\tau^{(i)}_{k} \in \partial T(C_{i,k})$ for each $1 \leq k \leq j_i$, then $\tau \in \partial F_\nu$. 
\end{lem}

\begin{proof}
	Suppose that the hypothesis of the lemma holds and, at the same time, that $\tau = (1-\lambda) \mu + \lambda \eta$ for $\mu =  (\mu_i;\alpha_k^{(i)},\mu_k^{(i)}),\eta =  (\eta_i;\alpha_k^{(i)},\eta_k^{(i)}) \in F_\nu$ and $\lambda \in (0,1)$. Then for every $i \in \mathbb N$, $\tau_i = (1-\lambda)\mu_i + \lambda \eta_i$ so that $\alpha_k^{(i)} \tau_k^{(i)} = (1-\lambda)\alpha_k^{(i)} \mu_k^{(i)} + \lambda \alpha_k^{(i)} \eta_k^{(i)}$ for each $1\leq k \leq j_i$. Hence, for sufficiently large $i$, we have $\tau_k^{(i)} = \mu_k^{(i)} = \eta_k^{(i)}$ for each $1 \leq k \leq j_i$ since $\tau_k^{(i)}$ is extreme. Thus $\tau = \mu = \eta$, and so $\tau \in \partial F_{\nu}$.
\end{proof}

\section{Main result}
	
\begin{thm}\label{T:main}
	Let $(G_i,\theta_i)_{i\in \mathbb N}$ be an inductive sequence of ordered abelian groups with distinguished order units, where $G_i = (\mathbb Z^{j_i},\mathbb Z^{j_i}_{+},(n_{i,1},\dotsc,n_{i,j_i}))$, and let $C_0$ be a noncommutative nuclear unital C*-algebra with more than one trace. If the canonical AF subalgebra $A$ of the noncommutative Villadsen algebra $B((G_i,\theta_i),C_0)$ is simple, then for any $\nu \in \partial T(A)$, the fiber over $\nu$ is the Poulsen simplex.
\end{thm}

\begin{proof}
	Since the fiber over $\nu$, $F_\nu$, is necessarily a simplex by Lemma \ref{L:fiberIsCompact}, it is sufficient to show that $\overline{\partial F_{\nu}} = F_{\nu}$. 
	
	Let $B= B((G_i,\theta_i),C_0)$, and suppose that $B$ is the limit of the inductive sequence $(B_i,\phi_i)_{i\in \mathbb N}$, with $B_{i,k} =M_{n_{i,k}}(C_0^{\otimes n_{i,k}})$ and $B_i = \bigoplus_{k=1}^{j_i} B_{i,k}$ (recall the form of $\phi_i$ from Equation \eqref{E:connectingMap}). Denote the multiplicity matrix for the composed map $\theta_{i,i+t-1}\colon G_i \to G_{i+t}$ by $[\theta_{i,i+t-1;k,l}]$, $1\leq k \leq j_i$, $1 \leq l \leq j_{i+t}$, $i,t \in \mathbb N$.
	
	Let $\nu = (\nu_i;\alpha^{(i)}_k,\text{Tr}^{(i)}_k)$ and $\tau = (\tau_i;\lambda_k^{(i)},\tau_k^{(i)}) \in F_{\nu}$, and let $\mathcal N = \mathcal N(\epsilon,\mathcal F) \subseteq F_\nu$ be a basic neighborhood of $\tau$.  We will show that $\mathcal N$ contains an extreme trace. 
	
	
	Without loss of generality, assume $\mathcal F \subset B_{i'}$ for some $i' \in \mathbb N$. For each $1\leq k \leq j_{i'}$, let $\mathcal F_{k}$ denote the subset of $B_{i',k}$ consisting of the $k$th component of each $b \in \mathcal F$. Consider the basic neighborhood $\mathcal N_k = \mathcal N_k(\epsilon,\mathcal F_k) \subseteq T(B_{i',k})$ of $\tau_k^{(i')}$. By Lemma \ref{L:kM}, there exists a number $N_k$ such that when $n > N_k$, there are $n$ points in $\partial T(B_{i',k})$ whose average is contained in $\mathcal N_k$. In fact, since $A$ is simple, there exists a $t' \in \mathbb N$ such that for each $1 \leq l \leq j_{i'+t'}$,
\begin{align*}
	\theta_{i',i'+t'-1;k,l} > N_k
\end{align*}
for each $1 \leq k \leq j_{i'}$. Thus, in particular, for each $1 \leq l \leq j_{i'+t'}$ there exist $\mu_{k,l,1},\dotsc,\mu_{\theta_{i',i'+t'-1;k,l}} \in \partial T(B_{i',k})$ such that for any $b \in\mathcal F_k$,
\begin{align}\label{E:inequality}
	\bigg | \frac{1}{\theta_{i',i'+t'-1;k,l}}\sum_{m=1}^{\theta_{i',i'+t'-1;k,l}} \mu_{k,l,m}(b) - \tau_k^{(i')}(b) \bigg | < \epsilon,
\end{align}
for each $1 \leq k \leq j_{i'}$. 

	Consider a trace $\eta = (\eta_i;\alpha_{k}^{(i)},\eta_{k}^{(i)}) \in F_\nu$ such that 
\begin{align*}
	\eta^{(i'+t')}_l:=  \bigotimes_{m=1}^{\theta_{i',i'+t'-1;1,l}} \mu_{1,l,m} \otimes \cdots \otimes \bigotimes_{m=1}^{\theta_{i',i'+t'-1;j_{i'},l}} \mu_{j_{i'},l,m}, \quad 1 \leq l \leq j_{i'+t'}.
\end{align*}
Then, for any $t \in \mathbb N$, a brief calculation reveals that the expression for $\eta_l^{(i'+t'+t)}$ satisfying the requirement that $\phi_{i'+t',i'+t'+t-1}^*(\eta_{i'+t'+t}) = \eta_{i'+t'}$ is
\begin{align*}
	\eta^{(i'+t'+t)}_l=  \bigotimes_{k=1}^{j_{i'+t'}} \big (\eta_{k}^{(i'+t')}\big )^{\otimes \theta_{i'+t',i'+t'+t-1;k,l}},\quad 1 \leq l \leq j_{i'+t'+t}
\end{align*}
(see Equations \eqref{E:traceImageParts}, \eqref{E:traceImageTraces}, and \eqref{E:coefficientImage}). 

	Because $C_0$ is nuclear, the tensor product of extreme traces is extreme (for this, see for example \cite[Proposition 11.3.2]{rvKjrR97}); thus $\eta_{l}^{(i'+t'+t)} \in \partial T(B_{i'+t'+t,l})$ for each $1 \leq l \leq j_{i'+t'+t}$, for every $t \in \mathbb N$. It then follows from Lemma \ref{L:limitOfExtremeTraces} that $\eta \in \partial F_\nu$. It is now sufficient to show that $\eta \in \mathcal N$. 

	In fact, another brief calculation (using Equations \eqref{E:traceImageParts}, \eqref{E:traceImageTraces}, and \eqref{E:coefficientImage}) shows that the expression for $\eta_k^{(i')}$ satisfying the equation $\phi_{i',i'+t'-1}^*(\eta_{i'+t'}) = \eta_{i'}$ is
\begin{align*}
	\eta_{k}^{(i')} = \frac{1}{\alpha_k^{(i)}}\sum_{l=1}^{j_{i'+t'}}   \sum_{m=1}^{\theta_{i',i'+t'-1;k,l}} \frac{\alpha_l^{(i'+t')} n_{i',k}}{n_{i'+t',l}} \mu_{k,l,m} ,\quad 1 \leq k \leq j_{i'}.
\end{align*}
It follows that for every $b = (b_1,\dotsc,b_{j_{i'}}) \in \mathcal F$,
\begin{align*}
	|\tau(b) - \eta(b)| &= |\tau_{i'}(b)-\eta_{i'}(b)|\\
	&=\Bigg |\sum_{k=1}^{j_{i'}}\alpha_{k}^{(i')}\big ( \tau_k^{(i')}(b_k) - \eta_k^{(i')}(b_k) \big ) \Bigg | \\
	&= \Bigg | \sum_{k=1}^{j_{i'}} \sum_{l=1}^{j_{i'+t'}} \frac{\alpha_{l}^{(i'+t')}n_{i',k}\theta_{i',i'+t'-1;k,l}}{n_{i'+t',l}} \Big ( \tau_{k}^{(i')}(b_k) - \frac{1}{\theta_{i',i'+t'-1;k,l}} \sum_{m=1}^{\theta_{i',i'+t'-1;k,l}}\mu_{k,l,m}(b_{k}) \Big ) \Bigg | \\
	&\leq \sum_{k=1}^{j_{i'}} \sum_{l=1}^{j_{i'+t'}} \frac{\alpha_{l}^{(i'+t')}n_{i',k}\theta_{i',i'+t'-1;k,l}}{n_{i'+t',l}} \Big | \tau_{k}^{(i')}(b_k) - \frac{1}{\theta_{i',i'+t'-1;k,l}} \sum_{m=1}^{\theta_{i',i'+t'-1;k,l}}\mu_{k,l,m}(b_{k}) \Big |  \\
	&< \sum_{k=1}^{j_{i'}} \sum_{l=1}^{j_{i'+t'}} \frac{\alpha_{l}^{(i'+t')}n_{i',k}\theta_{i',i'+t'-1;k,l}}{n_{i'+t',l}} \epsilon\\
	&= \epsilon,
\end{align*}
where the last inequality is a result of Equation \eqref{E:inequality} and the last equality is a result of the fact that $n_{i'+t',l} = \sum_{1\leq k \leq j_{i'}} n_{i',k}\theta_{i',i'+t'-1;k,l}$. Thus, $\eta \in \mathcal N$.
\end{proof}

	As an obvious corollary, we have:
\begin{cor}\label{C:main}
	Let $(G_i,\theta_i)_{i\in \mathbb N}$ and $C_0$ be as in the statement of Theorem \ref{T:main}. If the canonical AF subalgebra of the noncommutative Villadsen algebra $B((G_i,\theta_i),C_0)$ is simple and has a unique trace, then $T(B)$ is the Poulsen simplex.
\end{cor}

	Note that in the statement of Theorem \ref{T:main}, we must specify $C_0$ to have more than a single trace or else the result does not hold. If $C_0$ has a unique trace, then so too does $M_{n_{i,k}}(C_0^{\otimes n_{i,k}})$, which implies $F_\nu$ is a singleton (because there is only one possible sequence of tuples of scalars for any member of $F_\nu$). But (by convention) a point is not the Poulsen simplex. 

	In fact, in a special case, Theorem \ref{T:main} applies to the AF-Villadsen algebras of \cite{gaEzN25-Vill}. Let $(G_i,\theta_i)_{i\in \mathbb N}$ be as in the statement of Theorem \ref{T:main}, and let $C_0 = C(X)$ for a compact metrizable seed space $X$ which is not a single point. Consider the Villadsen algebra $B =B((G_{i},\theta_i),C_0)$ obtained as the limit of the inductive sequence $(B_i,\phi_i)_{i \in \mathbb N}$, where 
\begin{align*}
	B_i = \bigoplus_{k=1}^{j_i} B_{i,k}, \quad B_{i,k} = M_{n_{i,k}}\big (C_0^{\otimes n_{i,k}}\big ) = M_{n_{i,k}}\big ( C(X^{n_{i,k}}) \big )
\end{align*}
and where the multiplicity matrix for the composed map $\theta_{i,i+t-1}$ is given by $[\theta_{i,i+t-1;k,l}]$, $1\leq k \leq j_i$, $1 \leq l \leq j_{i+t}$, $t \in \mathbb N$. Since $C_0$ is commutative, the seed for $\phi_i$ has the form
\begin{multline}\label{E:connectingMapNew}
	\bigoplus_{k=1}^{j_i}C(X^{n_{i,k}}) \ni \bigoplus_{k=1}^{j_i} f_k \mapsto \\ \bigoplus_{l=1}^{j_{i+1}} \text{diag} \big ( f_1\circ \pi_{1,l,1},\dotsc,f_1\circ \pi_{1,l,\theta_{i,i;1,l}},\dotsc,f_{j_i} \circ \pi_{j_i,l,1},\dotsc,f_{j_i} \circ \pi_{j_i,l,\theta_{i,i;j_i,l}} \big )   \\ \in \bigoplus_{l=1}^{j_{i+1}} M_{\theta_{i,i;1,l}+\theta_{i,i;2,l}+\cdots+\theta_{i,i;j_i,l}}\big (C(X^{n_{i+1,l}})\big ),
\end{multline}
where $\pi_{k,l,m}$ is the projection of
\begin{align*}
	X^{n_{i+1,l}} = \underbrace{X^{n_{i,1}}\times \cdots \times X^{n_{i,1}}}_{\theta_{i,i;1,l}}\times \cdots \times \underbrace{X^{n_{i,j_i}}\times \cdots \times X^{n_{i,j_i}}}_{\theta_{i,i;j_i,l}}
\end{align*}
onto the $m$-th factor of $X^{n_{i,k}}$, for $1\leq m \leq \theta_{i,i;k,l}$, $1\leq k \leq j_i$, and $1 \leq l \leq j_{i+1}$.

	Now for each $i \in \mathbb N$, let $E_{i;k,l} \subseteq X^{n_{i,k}}$ be a (nonempty) finite point evaluation set, $1 \leq k \leq j_i$, $1 \leq l \leq j_{i+1}$, and let
\begin{align*}
	D_i = \bigoplus_{k=1}^{j_i} D_{i,k},\quad D_{i,k} =  M_{\Tilde{n}_{i,k}}\big (C(X^{n_{i,k}})\big ),
\end{align*}
where $\Tilde{n}_{1,l}=n_{1,l}$ and
\begin{align*}
	\Tilde{n}_{i+1,l} = \sum_{k=1}^{j_{i}}\big (\theta_{i,i;k,l}+|E_{i;k,l}| \big )\Tilde{n}_{i,k}.
\end{align*}
Fixing $i \in \mathbb N$, define the seed of an injective unital *-homomorphism $\psi_{i} \colon D_i \to D_{i+1}$ (up to unitary equivalence) by
\begin{multline*}
	\bigoplus_{k=1}^{j_i}C(X^{n_{i,k}}) \ni \bigoplus_{k=1}^{j_i} f_k \mapsto \\ \bigoplus_{l=1}^{j_{i+1}} \text{diag} \big ( f_1\circ \pi_{1,l,1},\dotsc,f_1\circ \pi_{1,l,\theta_{i,i;1,l}},f_1(E_{i;1,l}),\dotsc,f_{j_i} \circ \pi_{j_i,l,1},\dotsc,f_{j_i} \circ \pi_{j_i,l,\theta_{i,i;j_i,l}}, f_{j_i}(E_{i;j_i,l}) \big )   \\ \in \bigoplus_{l=1}^{j_{i+1}} M_{\theta_{i,i;1,l}+|E_{i;1,l}|+\theta_{i,i;2,l}+|E_{i;2,l}|+\cdots+\theta_{i,i;j_i,l}+|E_{i;j_i,l}|}\big (C(X^{n_{i+1,l}})\big ),
\end{multline*}
where $\pi_{k,l,m}$ is defined as in Equation \eqref{E:connectingMapNew}. Then define $D$ to be the limit of the inductive sequence $(D_i,\psi_i)_{i\in \mathbb N}$. 

	Denote the canonical AF subalgebra of $B$ by $A$; assume it is simple. For each $i \in \mathbb N$, letting $S(\mathbb Z^{j_i})$ denote the state space of $\mathbb Z^{j_i}$ normalized with respect to the order unit $(n_{i,1},\dotsc,n_{i,j_i})$, recall that $\textit{Aff}(S(K_0(A)))$ is given by the limit of the inductive sequence $(\textit{Aff}(S(\mathbb Z^{j_i})), \theta_i^{**})_{i\in \mathbb N}$, where 
\begin{align*}
	\theta_{i,i+t-1}^{**} = [\theta_{i,i+t-1;k,l}n_{i,k}/n_{i+t,l}]_{k,l}
\end{align*} 
and we make the canonical identification $\textit{Aff}(S(\mathbb Z^{j_i})) \cong \mathbb R^{j_i}$. Also, consider the map
\begin{align*}
	\Theta_{i}\colon \textit{Aff}(S(\mathbb Z^{j_i})) \to \textit{Aff}(S(\mathbb Z^{j_{i+1}})),\quad \Theta_i = [\Theta_{i;k,l}] = \big [ n_{i,k} (\theta_{i;k,l}+|E_{i;k,l}|)/n_{i+1,l} \big ]_{k,l}.
\end{align*}
It will be useful in the sequel to note that 
\begin{align*}
	\Theta_{i,i+t-1}-\theta_{i,i+t-1}^{**} = \big [|E_{i,i+t-1;k,l}|n_{i,k}/n_{i+t,l}\big ]_{k,l}
\end{align*}
is a positive map. 

	Then, for each $i \in \mathbb N$, identifying each element of $\textit{Aff}(S(\mathbb Z^{j_i}))$ with the appropriate element of $\textit{Aff}(S(K_0(A)))$, consider the sequence $(r_i)_{i\in \mathbb N}\subseteq \textit{Aff}(S(K_0(A)))$, where $r_1$ denotes the order unit of $\textit{Aff}(S(\mathbb Z^{j_1}))$ (i.e., the constant function equal to one, or equivalently, the vector of all ones in $\mathbb R^{j_1}$) and 
\begin{align*}
	r_i = \Theta_{1,i-1}(r_1) = \Theta_{i-1}(r_{i-1}), \quad i >1. 
\end{align*}
Then a short calculation shows that
\begin{align*}
	r_{i} = \Big (\frac{\Tilde{n}_{i,1}}{n_{i,1}},\dotsc,\frac{\Tilde{n}_{i,j_{i}}}{n_{i,j_{i}}}\Big ), \quad i \in \mathbb N,
\end{align*}
and the positivity of $\Theta_{i,i+t-1}-\theta_{i,i+t-1}^{**}$ implies
\begin{align*}
	r_{i+1}-r_i = (\Theta_{i}-\theta_{i}^{**})(r_i)  \geq 0
\end{align*}
so that $(r_i)$ is increasing. 

	We note in passing that the condition on the sequence $(r_i)_{i\in \mathbb N}$ specified in the following theorem is analogous to that which ensures the existence of the function $r_{\infty}^{(0)}$ from \cite{gaEzN25-Vill} as a nonzero member of $\textit{Aff}(S(K_0(A)))$. 

	
\begin{thm}\label{T:traceIntertwining}
	In the setting above, if the sequence $(r_i)_{i\in \mathbb N}$ converges uniformly (in the supremum norm) and the limit is a constant function, e.g., if $A$ has a unique trace, then $T(D) \cong T(B)$; if the (uniform) limit is not constant, then the tracial cones are isomorphic, i.e., $\mathbb R^+T(D) \cong \mathbb R^+T(B)$. 
\end{thm}

\begin{proof}
	We prove the first statement of the theorem first. It is enough to show that the diagrams
\begin{align}\label{E:diagram1}
	\textit{Aff}\big (T(B_1) \big ) \xrightarrow{\phi_1^{**}}  \textit{Aff}\big (T(B_2) \big ) \xrightarrow{\phi_2^{**}}\textit{Aff}\big (T(B_3) \big ) \xrightarrow{\phi_3^{**}} \cdots 
\end{align}
and 
\begin{align}\label{E:diagram2}
	\textit{Aff}\big (T(D_1) \big ) \xrightarrow{\psi_1^{**}}  \textit{Aff}\big (T(D_2) \big ) \xrightarrow{\psi_2^{**}}\textit{Aff}\big (T(D_3) \big ) \xrightarrow{\psi_3^{**}} \cdots
\end{align}
have an approximate intertwining as sequences of order unit Banach spaces when the sequence $(r_i)_{i\in \mathbb N}$ converges uniformly to a constant function. Note that when $A$ has a unique trace, $\textit{Aff}(S(K_0(A)))$ is a one-dimensional vector space so that, if $(r_i)_{i\in \mathbb N}$ converges uniformly, it necessarily converges to a constant function, i.e., a scalar multiple of the order unit. 

	For any $n \in \mathbb N$ and any compact Hausdorff space $Y$, we make the identification 
\begin{align*}
	\textit{Aff}\big (T\big (M_{n}(C(Y)) \big) \big ) \cong C_{\mathbb R}(Y)
\end{align*}
so that
\begin{align*}
	\textit{Aff}\big (T(B_i) \big ) \cong \bigoplus_{k=1}^{j_i} C_{\mathbb R}(X^{n_{i,k}}) \cong \textit{Aff}\big (T(D_i) \big ), \quad i \in \mathbb N
\end{align*}
(as order unit Banach spaces). It follows that the map 
\begin{align*}
	\phi_{i,i+t-1}^{**}\colon \bigoplus_{k=1}^{j_i} C_{\mathbb R}(X^{n_{i,k}}) \to \bigoplus_{k=1}^{j_{i+t}} C_{\mathbb R}(X^{n_{i+t,l}}), \quad t \in \mathbb N
\end{align*}
has the form
\begin{align*}
	\phi_{i,i+t-1}^{**}\big ((h_1,\dotsc,h_{j_i})\big ) = \bigoplus_{l=1}^{j_{i+t}} \sum_{k=1}^{j_i} \sum_{m=1}^{\theta_{i,i+t-1;k,l}} \frac{n_{i,k}}{n_{i+t,l}} h_k \circ \sigma_{k,l,m},
\end{align*}
where $\sigma_{k,l,m}$ is the projection of
\begin{align*}
	X^{n_{i+t,l}} = \underbrace{X^{n_{i,1}}\times \cdots \times X^{n_{i,1}}}_{\theta_{i,i+t-1;1,l}}\times \cdots \times \underbrace{X^{n_{i,j_i}}\times \cdots \times X^{n_{i,j_i}}}_{\theta_{i,i+t-1;j_i,l}}
\end{align*}
onto the $m$-th factor of $X^{n_{i,k}}$, $1\leq m \leq \theta_{i,i+t-1;k,l}$, $1\leq k \leq j_i$, $1 \leq l \leq j_{i+t}$. 

	Letting $\mu = \lambda_1 \mu_1+\cdots + \lambda_{j_{i+t}}\mu_{j_{i+t}} \in T(D_{i+t})$, notice that
\begin{align*}
	\psi_{i,i+t-1}^*(\mu) = \sum_{l=1}^{j_{i+t}} \sum_{k=1}^{j_i}\frac{\Tilde{n}_{i,k}\lambda_l}{\Tilde{n}_{i+t,l}}   \Big ( \sum_{m=1}^{\theta_{i,i+t-1;k,l}} \mu_{k,l,m} + \sum_{x \in E_{i,i+t-1;k,l}} \text{Tr}_{k,l,x} \Big ),
\end{align*}
where $\mu_{k,l,m}$ is defined as in Equation \eqref{E:traceImageParts}, $\text{Tr}_{k,l,x}(f) = f(x)$ for $f \in C(X^{n_{i,k}})$, and $E_{i,i+t-1;k,l} \subseteq X^{n_{i,k}}$ is a point evaluation set for the composed map $\psi_{i,i+t-1}$. It follows that the map $\psi_{i,i+t-1}^{**}$ has the form
\begin{align*}
	\psi_{i,i+t-1}^{**}\big ((h_1,\dotsc,h_{j_i})\big ) = \bigoplus_{l=1}^{j_{i+t}} \sum_{k=1}^{j_i} \frac{\Tilde{n}_{i,k}}{\Tilde{n}_{i+t,l}} \Big (\sum_{m=1}^{\theta_{i,i+t-1;k,l}}  h_k \circ \sigma_{k,l,m} + \sum_{x \in E_{i,i+t-1;k,l}} h_k(x) \Big ).
\end{align*}
Hence, writing $h = (h_1,\dotsc,h_{j_i})$, the $l$-th component of  $\psi_{i,i+t-1}^{**}(h)-\phi_{i,i+t-1}^{**}(h)$ is 
\begin{multline}\label{E:componentOfDifference}
	\big (\psi_{i,i+t-1}^{**}(h) - \phi_{i,i+t-1}^{**}(h)\big )_l = \\
	\sum_{k=1}^{j_i} \bigg ( \frac{\Tilde{n}_{i,k}}{\Tilde{n}_{i+t,l}}\sum_{m=1}^{\theta_{i,i+t-1;k,l}} h_k \circ \sigma_{k,l,m} + \frac{\Tilde{n}_{i,k}}{\Tilde{n}_{i+t,l}}\sum_{x \in E_{i,i+t-1;k,l}} h_k(x)  - \frac{n_{i,k}}{n_{i+t,l}} \sum_{m=1}^{\theta_{i,i+t-1;k,l}}h_k \circ \sigma_{k,l,m}\bigg ). 
\end{multline}

	
	Let $(r_i)_{i\in \mathbb N}$ converge to the constant function $r \mathbf1$, where $r \geq 1$ since $(r_i)$ is increasing. It follows that given a small parameter $\epsilon > 0$, for sufficiently large $i$, $|r_{i,k} - r| < \epsilon$ for each $1 \leq k \leq j_i$, where $r_{i,k}$ denotes the $k$-th component of $r_i$. Furthermore, also because $(r_i)$ is increasing, $r_{i,k} \leq r$ for each $1 \leq k \leq j_i$, $i\in \mathbb N$; hence for any $i \in \mathbb N$, there exists a $t \in \mathbb N$ such that $r_{i,k} \leq r_{i+t,l}$ for each $1 \leq k \leq j_i$ and $1 \leq l \leq j_{i+t}$. 
	
	From Equation \eqref{E:componentOfDifference}, we now see that, for $\|h\| \leq 1$, 
\begin{align*}
	\big \|  \big (\psi_{i,i+t-1}^{**}(h) - \phi_{i,i+t-1}^{**}(h)\big )_l \big \| \leq \sum_{k=1}^{j_i} \bigg | \frac{r_{i,k}}{r_{i+t,l}} -1 \bigg | \frac{n_{i,k}}{n_{i+t,l}} \theta_{i,i+t-1;k,l} + \sum_{k=1}^{j_i}\frac{\Tilde{n}_{i,k}}{\Tilde{n}_{i+t,l}}|E_{i,i+t-1;k,l}|
\end{align*}
so that, for sufficiently large $t$ and $\|h\| \leq 1$, 
\begin{align}\label{E:differenceOfFunctions}
	\big \|  \big (\psi_{i,i+t-1}^{**}(h) - \phi_{i,i+t-1}^{**}(h)\big )_l \big \| &\leq \sum_{k=1}^{j_i} \bigg ( 1-\frac{r_{i,k}}{r_{i+t,l}} \bigg )\frac{n_{i,k}}{n_{i+t,l}} \theta_{i,i+t-1;k,l} + \sum_{k=1}^{j_i}\frac{\Tilde{n}_{i,k}}{\Tilde{n}_{i+t,l}}|E_{i,i+t-1;k,l}|\\  \notag &= 1-\sum_{k=1}^{j_i}\frac{\Tilde{n}_{i,k}}{\Tilde{n}_{i+t,l}}\theta_{i,i+t-1;k,l}+ \sum_{k=1}^{j_i}\frac{\Tilde{n}_{i,k}}{\Tilde{n}_{i+t,l}}|E_{i,i+t-1;k,l}|\\ \notag &= 2\sum_{k=1}^{j_i}\frac{r_{i,k}}{r_{i+t,l}}\frac{{n}_{i,k}}{{n}_{i+t,l}}|E_{i,i+t-1;k,l}|\\ \notag &\leq 2\sum_{k=1}^{j_i}\frac{{n}_{i,k}}{{n}_{i+t,l}}|E_{i,i+t-1;k,l}|, \notag
\end{align}
where we have used the identities 
\begin{align*}
	\sum_{k=1}^{j_i} n_{i,k} \theta_{i,i+t-1;k,l} = n_{i+t,l}, \quad \sum_{k=1}^{j_i} \Tilde{n}_{i,k} \big (\theta_{i,i+t-1;k,l}+|E_{i,i+t-1;k,l}|\big) = \Tilde{n}_{i+t,l}
\end{align*}
and the fact that $r_{i,k} \leq r_{i+t,l}$ for each $1 \leq k \leq j_i$ and $1 \leq l \leq j_{i+t}$. 

	Now we claim that 
\begin{align}\label{E:fundamentalLimit}
	\lim_{i\to\infty} \lim_{t\to\infty} \sum_{k=1}^{j_i} \frac{n_{i,k}}{n_{i+t,l}}|E_{i,i+t-1;k,l}| = 0. 
\end{align}
To see this, start by considering the sequence $(g_i^{(s)})_{i\in \mathbb N} \subseteq \textit{Aff}(S(K_0(A)))$ obtained by replacing the first $s$ terms of the sequence $(r_i)_{i\in \mathbb N}$ by the order units; specifically, let $(g_i^{(1)})_{i\in \mathbb N} = (r_{i})_{i\in \mathbb N}$, and for each $s > 1$, let $g_1^{(s)} = r_1$ and 
\begin{align*}
	g_{i}^{(s)} = \begin{cases}	\theta^{**}_{1,i-1}(r_1) = \theta^{**}_{i-1}(g_{i-1}^{(s)}),& 1 < i \leq s\\
	\Theta_{s,i-1}(g_s^{(s)}) = \Theta_{i-1}(g_{i-1}^{(s)}),&  i > s
	\end{cases}.
\end{align*}
By the positivity of $\Theta_{i,i+t-1}-\theta_{i,i+t-1}^{**}$ (and that of $\Theta_{i,i+t-1}$ itself), we see for any $s \in\mathbb N$,
\begin{align*}
	g_{i}^{(1)} - g_i^{(s)} = \Theta_{s,i-1}\big ( (\Theta_{1,s-1}-\theta_{1,s-1}^{**})  (r_1) \big ) \geq 0
\end{align*}
so that $g_{i}^{(1)} \geq g_i^{(s)}$ for each $i \in \mathbb N$. Then, because 
\begin{align*}
	\lim_{i\to\infty} g_i^{(1)} = g_1 = r\mathbf 1 \in \textit{Aff}(S(K_0(A)))
\end{align*} 
by hypothesis and $(g_i^{(s)})_{i\in \mathbb N}$ is increasing (apply the positive map $\Theta_i-\theta_i^{**}$ to $g_{i}^{(s)}$ to see this), $(g_i^{(s)})_{i\in \mathbb N}$ converges to an element $g_s \in \textit{Aff}(S(K_0(A)))$ such that $g_s \leq g_1$. 



	Furthermore, the sequence $(g_s)_{s\in \mathbb N}$ converges to the order unit $\mathbf 1 \in \textit{Aff}(S(K_0(A)))$. Indeed, fixing $s \in \mathbb N$, identify $g_s - \mathbf 1$ with the sum 
\begin{align}\label{E:differenceSum}
	\sum_{i=1}^\infty (g_{i+1}^{(s)} - g_{i}^{(s)}) = \sum_{i=s}^\infty (g_{i+1}^{(s)} - g_{i}^{(s)}).
\end{align}
Notice for $i > s$, $g_{i+1}^{(s)} - g_i^{(s)} \leq g_{i+1}^{(1)} - g_i^{(1)}$ since
\begin{align*}
	g_{i+1}^{(1)} - g_i^{(1)} - (g_{i+1}^{(s)} - g_i^{(s)}) =(\Theta_i-\theta_i^{**}) \big ( \Theta_{s,i-1} ( (\Theta_{1,s-1}-\theta^{**}_{1,s-1})(r_1)) \big )\geq 0;
\end{align*}
a similar formula shows $g_{s+1}^{(s)} - g_s^{(s)} \leq g_{s+1}^{(1)} - g_s^{(1)}$. Now given $\epsilon > 0$, let $s'$ be such that 
\begin{align*}
\big \|\sum_{i=s'}^{\infty} g_{i+1}^{(1)} - g_{i}^{(1)}\big \| < \epsilon. 
\end{align*} 
Then by Equation \eqref{E:differenceSum} and the fact that $g_{i+1}^{(s')} - g_i^{(s')} \leq g_{i+1}^{(1)} - g_i^{(1)}$ for $i \geq s'$, we have $\|g_{s'}-\mathbf 1\| < \epsilon$. 

	
	Hence, for any $\epsilon > 0$, there exists $i \in \mathbb N$ such that
\begin{align*}
	\epsilon > \|g_i - \mathbf 1\| = \lim_{t\to\infty}  \| g_{i+t-1}^{(i)} - \theta^{**}_{1,i+t-1}(r_1) \| = \lim_{t\to \infty}\|  (\Theta_{i,i+t-1}- \theta^{**}_{i,i+t-1})(g_{i}^{(i)}) \|, 
\end{align*}
and since $g_{i}^{(i)}$ is just the order unit for $\textit{Aff}(S(\mathbb Z^{j_i}))$,
\begin{align*}
	\big ((\Theta_{i,i+t-1}- \theta^{**}_{i,i+t-1})(g_{i}^{(i)})\big )_l = \sum_{k=1}^{j_i} \frac{n_{i,k}}{n_{i+t,l}}|E_{i,i+t-1;k,l}|
\end{align*}
so that Equation \eqref{E:fundamentalLimit} holds. 

	Thus, it follows from Equation \eqref{E:differenceOfFunctions} that there is a sequence of natural numbers $(s_i)_{i\in \mathbb N}$ such that
\begin{align*}
	\big \| \psi_{s_i,s_{i+2}-1}^{**}(h) - \phi_{s_i,s_{i+2}-1}^{**}(h)\big \| < 2^{-i}, \quad i =2k-1,\ k \in \mathbb N, \quad \|h\| \leq 1
\end{align*} 
so that the diagram 
\[\begin{tikzcd}
	{\textit{Aff}\big (T(B_{s_1}) \big )} & {\textit{Aff}\big (T(B_{s_2})\big )} & {\textit{Aff}\big (T(B_{s_3})\big )} & \cdots & \\
	& {\textit{Aff}\big (T(D_{s_2}) \big )} & {\textit{Aff}\big (T(D_{s_3}) \big )} & {\textit{Aff}\big (T(D_{s_4}) \big )} & \cdots
	\arrow["{\phi_{s_1,s_2-1}^{**}}", from=1-1, to=1-2]
	\arrow["{\psi_{s_1,s_2-1}^{**}}"{pos=0.6}, shift right, from=1-1, to=2-2]
	\arrow["{\phi_{s_2,s_3-1}^{**}}", from=1-2, to=1-3]
	\arrow["{\phi_{s_3,s_4-1}^{**}}", from=1-3, to=1-4]
	\arrow["{\psi_{s_3,s_4-1}^{**}}"{pos=0.7}, from=1-3, to=2-4]
	\arrow["{\psi_{s_2,s_3-1}^{**}}"{pos=0.3}, shift right, from=2-2, to=1-3]
	\arrow["{\psi_{s_2,s_3-1}^{**}}"', from=2-2, to=2-3]
	\arrow["{\psi_{s_3,s_4-1}^{**}}"', from=2-3, to=2-4]
	\arrow["{\psi_{s_4,s_5-1}^{**}}"', from=2-4, to=2-5]
\end{tikzcd}\]
approximately commutes. We have now shown that the diagrams \eqref{E:diagram1} and \eqref{E:diagram2} have an approximate intertwining so that $T(B) \cong T(D)$ in the case that the sequence $(r_i)$ converges to a constant function. 

	To prove the second statement of the theorem, it is enough to show that the diagrams \eqref{E:diagram1} and \eqref{E:diagram2} have an approximate intertwining as sequences of ordered Banach spaces (i.e., we no longer assume that the intertwining morphisms preserve order units). 
	
	By Equation \eqref{E:fundamentalLimit}, there is a sequence of natural numbers $(s_i)_{i\in \mathbb N}$ such that 
\begin{align}\label{E:inequalityForIntertwining}
	\sum_{k=1}^{j_{s_i}} \frac{n_{s_i,k}}{n_{s_{i+1},l}}|E_{s_i,s_{i+1}-1;k,l}| < 2^{-i}\|r\|^{-1},\quad 1 \leq l \leq j_{s_{i+1}}. 
\end{align}
Now, for each $i \in \mathbb N$, consider the isomorphism of ordered Banach spaces
\begin{align*}
	\Delta_i\colon \textit{Aff}(T(B_{s_i}) ) \to \textit{Aff}(T(D_{s_i}) ),\quad  (h_1,\dotsc,h_{s_i}) \mapsto (r_{s_i,1}^{-1}h_1,\dotsc,r_{s_i,j_{s_i}}^{-1}h_{s_i})
\end{align*}
with $\Delta_i^{-1}(h_1,\dotsc,h_{s_{i}}) = (r_{s_i,1}h_1,\dotsc,r_{s_i,j_{s_i}}h_{s_i})$. Then for $\|h\| \leq 1$, a calculation shows that 
\begin{align*}
	\big \|\big (\Delta_{i+1}\big (\phi_{s_i,s_{i+1}-1}^{**}(h) \big ) - \psi_{s_i,s_{i+1}-1}^{**}\big ( \Delta_i (h)\big ) \big )_{l}\big \| &= \big \|\sum_{k=1}^{j_{s_i}} \sum_{x \in E_{s_i,s_{i+1}-1;k,l}} \frac{n_{s_i,k}}{\Tilde{n}_{s_{i+1},l}}h_k(x) \big \| \\ &\leq \sum_{k=1}^{j_{s_i}} \frac{n_{s_i,k}}{\Tilde{n}_{s_{i+1},l}}|E_{s_i,s_{i+1}-1;k,l}|\\ &\leq \sum_{k=1}^{j_{s_i}} \frac{n_{s_i,k}}{{n}_{s_{i+1},l}}|E_{s_i,s_{i+1}-1;k,l}|\\ &< 2^{-i}\|r\|^{-1} < 2^{-i},
\end{align*}
where the second last inequality follows from Equation \eqref{E:inequalityForIntertwining}; similarly, 
\begin{align*}
	\big \|\big (\Delta_{i+1}^{-1}\big (\psi_{s_i,s_{i+1}-1}^{**}(h) \big ) - \phi_{s_i,s_{i+1}-1}^{**}\big ( \Delta_i^{-1} (h)\big ) \big )_{l}\big \| &= \big \|\sum_{k=1}^{j_{s_i}} \sum_{x \in E_{s_i,s_{i+1}-1;k,l}} \frac{\Tilde{n}_{s_i,k}}{{n}_{s_{i+1},l}}h_k(x) \big \| \\ &\leq \sum_{k=1}^{j_{s_i}} \frac{\Tilde{n}_{s_i,k}}{{n}_{s_{i+1},l}}|E_{s_i,s_{i+1}-1;k,l}| \\ &= \sum_{k=1}^{j_{s_i}} r_{s_i,k} \frac{{n}_{s_i,k}}{{n}_{s_{i+1},l}}|E_{s_i,s_{i+1}-1;k,l}| \\  &\leq \|r\|\sum_{k=1}^{j_{s_i}}  \frac{{n}_{s_i,k}}{{n}_{s_{i+1},l}}|E_{s_i,s_{i+1}-1;k,l}| \\ &< 2^{-i}, 
\end{align*}
where the second last inequality follows from the fact that $(r_i)$ is increasing. We thus have the desired intertwining. 
\end{proof}

	We can now apply Theorem \ref{T:main} to certain AF-Villadsen algebras from \cite{gaEzN25-Vill}. 

\begin{cor}\label{C:simpleAFVill}
	In the setting above, if the sequence $(r_i)_{i\in \mathbb N}$ converges uniformly and $A$ has a unique trace, then $T(D)$ is the Poulsen simplex. 
\end{cor}

	Finally, we note that Theorem \ref{T:main} holds for nonnuclear seed algebras as well. Indeed, we only used the nuclearity of $C_0$ to justify that the product of extreme traces is extreme, in particular that (using the notation in the proof of Theorem \ref{T:main}) the trace
\begin{align*}
	\eta^{(i'+t'+t)}_l=  \bigotimes_{k=1}^{j_{i'+t'}} \big (\eta_{k}^{(i'+t')}\big )^{\otimes \theta_{i'+t',i'+t'+t-1;k,l}}
\end{align*}
is extreme for each $1 \leq l \leq j_{i'+t'+t}$, for every $t \in \mathbb N$. However, it was shown in \cite[Theorem 2.1]{cIdK19} that for any unital C*-algebras $A_1$ and $A_2$ and extreme traces $\tau_1 \in T(A_1)$ and $\tau_2 \in T(A_2)$, $\tau_1\otimes_{\beta} \tau_2$ is extreme in $T(A_1\otimes_\beta A_2)$ for any C*-norm $\beta$. We neglected to consider arbitrary unital seed algebras above because in this case the notational complexity needed to make our constructions precise would hinder one's understanding of our main result (i.e., what is lost due to notational complexity outweighs what might be gained by an increase in generality). 

\printbibliography

\end{document}